\documentclass{article}
\usepackage{graphicx}
\usepackage{amsmath}
\usepackage{amssymb}
\usepackage{booktabs}
\usepackage{multirow}
\usepackage{array}
\usepackage{amsfonts}
\usepackage{mathtools}
\usepackage{graphicx,pstricks,listings,stackengine}
\usepackage{ntheorem}
\usepackage{graphicx}
\usepackage{subfigure}
\usepackage{authblk}
\usepackage{todonotes}
\usepackage{hyperref}
 \hypersetup{
    colorlinks=true,       
    linkcolor=red,          
     citecolor=blue,        
     urlcolor=cyan           
 }
 
\usepackage{cite}
\usepackage{float}
\usepackage{geometry}
\newtheorem{proposition}{Proposition}
\newtheorem{lemma}{Lemma}
\newtheorem{theorem}{Theorem}
\newtheorem*{proof}{Proof}

\newtheorem{remark}{Remark}

\title{Critical Points, Stability, and Basins of Attraction of Three Kuramoto Oscillators with Isosceles Triangle Network}
 \author[a,b]{Xiaoxue Zhao}
 \author[a]{Xiang Zhou \thanks{Corresponding author: xizhou@cityu.edu.hk}}
 \affil[a]{School of Data Science and Department of Mathematics, City University of Hong Kong, Hong Kong SAR}
 \affil[b]{School of Mathematics, Harbin Institute of Technology, Heilongjiang, China}

 \date{}

\begin{document}

\maketitle

\begin{abstract}
    We investigate  the Kuramoto model with three oscillators   interconnected by an isosceles triangle network. The characteristic of this model is that the coupling connections between the oscillators can be either attractive or repulsive. We list all critical points and investigate their stability. We furthermore present a framework studying convergence towards stable critical points under special coupled strengths.
    The main tool is the linearization and the monotonicity arguments of oscillator diameter. 
\end{abstract}


{\section{Introduction}}
Synchronized collective phenomena are extensively studied in complex physical and biological systems.
One of the most successful models in the interpretation of the phenomena is the Kuramoto model \cite{kuramoto1984chemical}, which has a wide range of applications \cite{antonsen2008external,li2020synchronization}. 
{Flame flickering is a common phenomenon observed in combustion systems, where flames exhibit oscillatory behavior due to various instabilities. Understanding and controlling these oscillations are crucial for improving combustion efficiency and reducing emissions.
The Kuramoto model can be applied to flame flickering dynamics by modeling the flames as coupled oscillators\cite{chi2024synchronization, kitahata2009oscillation,gergely2020flickering}. Each flame element is  considered an oscillator with   the interactions   modeled as coupling terms. 
We are interested   in using the Kuramoto model    as a powerful theoretic framework for understanding the synchronization and phase dynamics of flame flickering.}

{The literature \cite{chi2024synchronization} has conceptually proposed 
 the Kuramoto model that involves three oscillators interconnected through an isosceles triangle network:}
\begin{equation}\label{mod1}
    \left\{\begin{array}{cc}
       \dot{\theta}_1(t)=\frac{K_2}{3}\sin(\theta_3(t)-\theta_1(t))+\frac{K_1}{3}\sin(\theta_2(t)-\theta_1(t))\\
       \dot{\theta}_2(t)=\frac{K_2}{3}\sin(\theta_3(t)-\theta_2(t))+\frac{K_1}{3}\sin(\theta_1(t)-\theta_2(t))\\
       \dot{\theta}_3(t)=\frac{K_2}{3}\sin(\theta_1(t)-\theta_3(t))+\frac{K_2}{3}\sin(\theta_2(t)-\theta_3(t)).
    \end{array}
    \right.
\end{equation}
{The flame states   are modelled as oscillators $\theta_i(t)\in\mathbb{R}, i=1,2,3$. Here the coupling strengths between $\theta_1$-$\theta_3$ and between $\theta_2$-$\theta_3$ are designed as the same $K_2$, so that all three-body coupling strengths are assumed by only two nonzero real numbers  $K_1$ and $K_2$. }
Our interest lies in the theoretical analysis of the synchronous dynamics of this system. 
It is evident that if $\Theta(t)=(\theta_1(t),\theta_2(t),\theta_3(t))$ is a solution of \eqref{mod1}, then for any constant $c\in\mathbb{R}$, $\Theta(t)+c(1,1,1)$ is also a solution. This is called translational invariance of the solution.

The dynamical system \eqref{mod1} is a gradient system $\dot{\Theta}=-\nabla V(\Theta)$
with the energy function {$V(\Theta)=-\frac{1}{3}[K_2\cos(\theta_3-\theta_1)+K_1\cos(\theta_2-\theta_1)+K_2\cos(\theta_2-\theta_3)].$}
It is easy to see that $\nabla V(\Theta)$ is continuous and $  \nabla^2V(\Theta) $ is continuous and uniformly bounded on $\mathbb{R}^3$, then $\nabla V(\Theta)$ is globally Lipschitz in $\Theta$ on $\mathbb{R}^3$ for any time $t\ge 0$. Hence, this gradient system with any initial data $\Theta(0)$ has a unique solution over $[0,\infty)$. 
Calling $\Theta^*\in\mathbb{R}^3$ \textit{a critical point} of system \eqref{mod1} refers to $\nabla V(\Theta^*)=0$.
Li and Xue \cite[Theorem 1.1]{li2019convergence} have proved that $\lim_{t\to\infty}\Theta(t)=\Theta^*\in\{\Theta^*\mid \nabla V(\Theta^*)=0\}$ and $\lim_{t\to\infty}\dot{\Theta}(t)=0$ for any initial data $\Theta(0)$. 

There have been some similar studies on system \eqref{mod1}, see 
\cite{ha2012basin,rogge2004stability,zhao2018formation}. 
These results assume identical and positive coupling strengths $(K_1=K_2>0)$. 
However, for flame dynamics, there is competition for limited resources among combustion units, such as oxygen or combustible material. This means that the interactions between these units can lead to negative coupling forces. The study of negative coupling forces can restrict the propagation speed, spread range, or combustion efficiency of a flame.
As far as we know, there is no existing study on the dynamics \eqref{mod1} with non-identical and negative coupling strengths, which is the main focus of our research.


\textbf{Contributions.} 
In this paper, we will perform rigorous analysis for system \eqref{mod1} and the main results are twofold. Firstly, we compute all critical points and analyze their stability, based on the linearization (see Lemma \ref{lem1-existence} and Proposition \ref{lem1-stability}).
Our results indicate that for certain coupling strengths, stable critical points can coexist, implying richer dynamical properties compared to those described in reference \cite{zhao2018formation}.
Secondly, we analyze the basins of attraction of critical points. Specifically, we focus on a special case: $K_1=-K_2<0$, where two stable critical points coexist. We provide a framework for determining the basins of attraction of these two stable critical points using novel inequalities (see Theorem \ref{Thm5Star}). 

\textbf{Organization of paper.} In Section \ref{sec2}, we identify the formation and stability of critical points  of system \eqref{mod1}. In Section \ref{sec3}, we prove the convergence of the solution and give estimates on the basins of critical points. Section \ref{sec4} is devoted to be a brief summary of this paper.

{\section{Critical points and stability }\label{sec2}}
This section concerns the formation and stability of  critical points, which plays a central role in systems theory. 
The following lemma lists all critical points in the range of $[0,2\pi)^3$ for system \eqref{mod1}. 
The critical points discussed here are considered in an equivalence class sense. In other words, if $\Theta^*$ is a critical point of the system \eqref{mod1}, then $\Theta^*+2L\pi, L\in\mathbb{R}^3$ is also a critical point. This will not be reiterated further.
\begin{lemma}\label{lem1-existence}
   For any $K_1,K_2\in\mathbb{R}\setminus\{0\}$, there are four critical points of system \eqref{mod1}:
   \[ \Theta^*_1=(0,\pi,0),\quad \Theta^*_2=(0,\pi,\pi),\quad \Theta^*_3=(0,0,0),\quad \Theta^*_4=(0,0,\pi). \]
   In addition, if $|\frac{K_2}{2K_1}|\le 1$, then there are two more critical points:
   \[\Theta^*_5=(0,2\arccos(-\frac{K_2}{2K_1}),\arccos(-\frac{K_2}{2K_1})),\]
   and
   \[\Theta^*_6=(0,2\pi-2\arccos(-\frac{K_2}{2K_1}),2\pi-\arccos(-\frac{K_2}{2K_1})).\]
\end{lemma}
\begin{proof}
    The results can be easily obtained by setting the right-hand side of system \eqref{mod1} to zero.
\end{proof}

The proposition below gives the stability of critical points, based on linearization.
\begin{proposition}\label{lem1-stability}
   The linear stability of critical points $\{\Theta_i^*\}_{i=1}^{6}$ of system \eqref{mod1} are as follows $(\Theta^*_i,i=5,6$ only exist under the condition $|\frac{K_2}{2K_1}|\le 1)$:
   \begin{enumerate}
       \item[$(1)$] $\Theta^*_1$ and $\Theta^*_2$ are unstable for any $K_1,K_2\in\mathbb{R}\setminus\{0\}$;
       \item[$(2)$] $\Theta^*_3$ is stable and all others are unstable if $K_2>0$ and $2K_1+K_2>0$; 
      \item[$(3)$]  $\Theta^*_4$ is stable and all others are unstable if $K_2<0$ and $2K_1-K_2>0$;
        \item[$(4)$] $\Theta^*_5$ and $\Theta^*_6$ are stable and all four others are unstable if  $2K_1+K_2<0$ and $2K_1-K_2<0$.
   \end{enumerate}
\end{proposition}
\begin{proof}
 The Jacobian matrix $J_{\Theta^*}$ at $\Theta^*=(\theta^*_1,\theta^*_2,\theta^*_3)$ is
\[J_{\Theta^*}=\frac{1}{3}\begin{bmatrix}
* & K_1\cos(\theta^*_2-\theta^*_1)& K_2\cos(\theta^*_3-\theta^*_1)\\
K_1\cos(\theta^*_1-\theta^*_2) & * & K_2\cos(\theta^*_3-\theta^*_2)\\
K_2\cos(\theta^*_1-\theta^*_3)& K_2\cos(\theta^*_2-\theta^*_3) & *
\end{bmatrix}\]
where $J_{ii}=-\sum_{j=1,j\neq i}^{3}J_{ij}$, are marked in ``$*$''. $J_{\Theta^*}$ has a simple eigenvalue $0$, due to the shift invariance. So, as long as the remaining two eigenvalues are negative, the critical point $\Theta^*$ is stable.The calculations of following eigenvalues and eigenvectors are trivial verification.
 \begin{enumerate}
    \item[$(a)$] The eigenvalues and corresponding eigenvectors of $\Theta^*_1$ and $\Theta^*_2$ are
    $\lambda_1=0,$ $ \lambda_2=\frac{K_1+\sqrt{K_1^2+3K^2_2}}{3},$ $ \lambda_3=\frac{K_1-\sqrt{K_1^2+3K^2_2}}{3},$
    and 
    \[\nu_1=\begin{bmatrix}
        1\\1\\1
    \end{bmatrix},\quad \nu_2=\begin{bmatrix}
        -\frac{K_1+\sqrt{K_1^2+2K^2_2}}{K_2}-1\\ \frac{K_1+\sqrt{K_1^2+2K^2_2}}{K_2}-1\\2
    \end{bmatrix},\quad \nu_3=\begin{bmatrix}
        -\frac{K_1-\sqrt{K_1^2+2K^2_2}}{K_2}-1\\ \frac{K_1-\sqrt{K_1^2+2K^2_2}}{K_2}-1\\2
    \end{bmatrix}. \]
    \item[$(b)$] The eigenvalues and corresponding eigenvectors of  $\Theta^*_3$ are  $\lambda_1=0,\lambda_4=-\frac{2K_1+K_2}{3},$ $ \lambda_5=-K_2,$
    and 
    $\nu_1=\begin{bmatrix}
        1\; 1\; 1
    \end{bmatrix}^{\mathrm{T}},$ $ \nu_4=\begin{bmatrix}
        -1\; 1\; 0
    \end{bmatrix}^{\mathrm{T}},$ $ \nu_5=\begin{bmatrix}
       -\frac{1}{2}\; -\frac{1}{2}\; 1
    \end{bmatrix}^{\mathrm{T}}. $
    \item[$(c)$] The eigenvalues and corresponding eigenvectors of $\Theta^*_4$ are
    $\lambda_1=0,$ $ \lambda_6=\frac{K_2-2K_1}{3},$ $ \lambda_7=K_2,   $
     and $\nu_1,\nu_4,\nu_5$.
     \item[$(d)$] The eigenvalues and corresponding eigenvectors of $\Theta^*_5$ and $\Theta^*_6$ are
    $\lambda_1=0,$ $ \lambda_8=\frac{4K_1^2-K_2^2}{6K_1},$ $ \lambda_9=\frac{K_2^2}{2K_1},   $
    and $\nu_1,\nu_4,\nu_5$.
\end{enumerate}
\end{proof}


Figure \ref{Fig1} demonstrates the above results. 
\begin{figure}[H]
    \centering
     \subfigure[$K_1>0$]{\includegraphics[
     width=0.4\textwidth
     ]{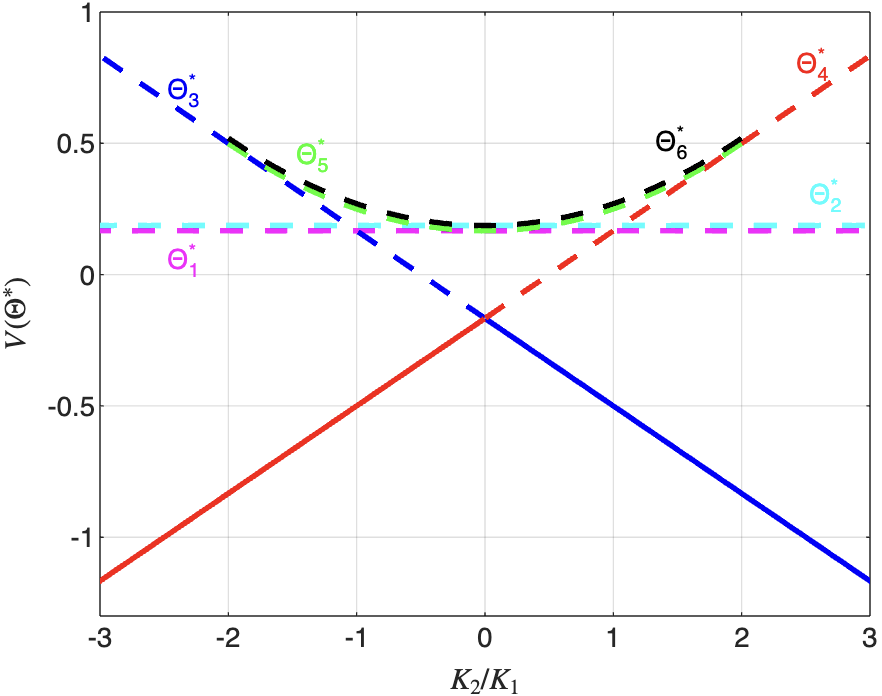}}
    \subfigure[$K_1<0$]{\includegraphics[ width=0.4\textwidth]{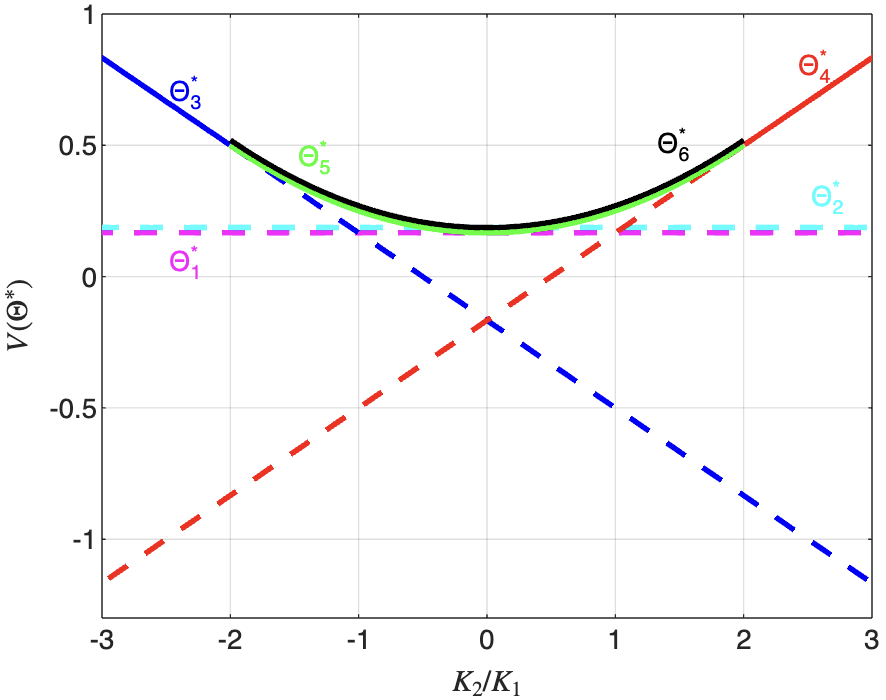}}
    \vspace{-0.3cm}
    \caption{ The six critical points $\{\Theta^*_i\}_{i=1}^6$ (marked in six different colors)
    with dependence on the value of coupling strengths. 
The plot is  the potential function $V$ of the critical points in terms of the ratio $K_2/K_1$. 
The stability of each critical point 
is depicted by the pattern of the lines:
the solid lines represent stability and the dashed lines represent instability. }
    \label{Fig1}
\end{figure}

{\section{Analysis of basins of attraction}\label{sec3}}
From Figure \ref{Fig1}, the critical points $\Theta^*_1$ and $\Theta^*_2$ are unstable for any $K_1$ and $K_2$, so the basins of attraction are Lebesgue zero-measure sets. {For $(K_1,K_2)\in \left((0,\infty)\times(0,\infty)\right)\cup \left((-\infty,0)\times (-2K_1,\infty)\right)$, $\Theta^*_3$ is the only stable point and   the solution of system \eqref{mod1} starting from almost any initial value   will   converge to   $\Theta^*_3$. The analysis for $\Theta^*_4$ is the same as  that of $\Theta^*_3$ under the condition $(K_1,K_2)\in \left((0,\infty)\times(-\infty,0)\right)\cup \left((-\infty,0)\times (-\infty,2K_1)\right)$.}
For $(K_1,K_2)\in (-\infty,0)\times(2K_1,-2K_1),$ the critical points $\Theta^*_5$ and $\Theta^*_6$ are both stable, while all the others are unstable. 

{We  will focus on the theoretic analysis of basins of attraction of co-existing stable points $\Theta^*_5$ and $\Theta^*_6$ when $(K_1,K_2)\in (-\infty,0)\times(2K_1,-2K_1)$, i.e., $K_1<0$ and $|K_2|< 2|K_1|$. 
 We recognize that the ratio $|K_2|/|K_1|$
  might  lead to substantial alterations in the shape and size of basins of attraction, making it difficult to accurately estimate by theory.
However, 
utilizing the phase diameter function approach, a sufficient condition for the initial condition within one of the basins is available, provided that a specific assumption of $K_1=-K_2$ is met. 
 On this assumption, 
$$\Theta^*_5=(0, \frac{2\pi}{3},\frac{\pi}{3}), \quad \Theta^*_6=(0, \frac{4\pi}{3},\frac{5\pi}{3}).$$ 
 Our primary analytical result is Theorem \ref{Thm5Star}, derived in the content that follows. In the absence of the $K_1=-K_2$ assumption, there are difficulties within this approach because the estimate in Lemma \ref{lemderi} becomes  nontrivial or perhaps impossible.
}

We   consider $\Theta^*_5$ first since the case for $\Theta^*_6$ is similar. 
Performing the translation $\tilde{\Theta}:=\Theta-\Theta^*_5$, i.e.,
$\tilde{\theta}_1=\theta_1,\tilde{\theta}_2=\theta_2-\frac{2\pi}{3},\tilde{\theta}_3=\theta_3-\frac{\pi}{3},$
we can rewrite the system \eqref{mod1} as
\begin{equation}\label{equtil}
    \left\{\begin{array}{cc}
       \dot{\tilde{\theta}}_1(t)=\frac{K_2}{3}\sin(\tilde{\theta}_3(t)-\tilde{\theta}_1(t)+\frac{\pi}{3})-\frac{K_1}{3}\sin(\tilde{\theta}_2(t)-\tilde{\theta}_1(t)-\frac{\pi}{3})\\
       \dot{\tilde{\theta}}_2(t)=\frac{K_2}{3}\sin(\tilde{\theta}_3(t)-\tilde{\theta}_2(t)-\frac{\pi}{3})-\frac{K_1}{3}\sin(\tilde{\theta}_1(t)-\tilde{\theta}_2(t)+\frac{\pi}{3})\\
       \dot{\tilde{\theta}}_3(t)=\frac{K_2}{3}\sin(\tilde{\theta}_1(t)-\tilde{\theta}_3(t)-\frac{\pi}{3})+\frac{K_2}{3}\sin(\tilde{\theta}_2(t)-\tilde{\theta}_3(t)+\frac{\pi}{3})
    \end{array}
    \right.
\end{equation}
and the problem at hand is transformed into estimating an attraction domain of $(0,0,0)$.

The challenge in studying system \eqref{equtil} lies in the fact that the signs before $\frac{\pi}{3}$, referred to as the phase-lag or frustration \cite{hsia2020synchronization,ha2014asymptotic,ha2014large}, can be both positive and negative. There are studies that focus on two situations: no frustration and consistently positive frustration\cite{ha2012basin,hsia2020synchronization,ha2014asymptotic,zhao2018formation,ha2014large}. The commonly used method is to construct a phase diameter function, estimate its derivative, and apply the Gr\"{o}nwall's inequality to prove that the phase diameter function will exponentially decay to zero. This paper adopts the same approach, but we provide new estimation inequalities for the frustration. 
Consider the phase diameter function defined by
$$\mathcal{D} (\tilde\Theta(t)):=\max_{i,j=1,2,3}(\tilde\theta_i(t)-\tilde\theta_j(t)),$$ 
which is non-negative, continuous and piece-wise differentiable with respect to time $t$. {Let $M_t=\operatorname{argmax}\{\dot{\tilde\theta}_{i}(t)\mid i\in \underset{l=1,2,3}{\operatorname{argmax}}~\tilde{\theta}_l\}$ and $m_t=\operatorname{argmin}\{\dot{\tilde\theta}_{i}(t)\mid i\in \underset{l=1,2,3}{\operatorname{argmin}}~\tilde{\theta}_l\}$, then 
by \cite[Lemma 2.2]{lin2007state}, the upper Dini derivative $D^+\mathcal{D}(\tilde\Theta(t))$ along the system \eqref{equtil} is given by
\[ D^+\mathcal{D}(\tilde\Theta(t)):=\limsup_{h\downarrow 0}\frac{\mathcal{D}(\tilde\Theta(t+h))-\mathcal{D}(\tilde\Theta(t))}{h}=\max_{i^{'}\in M_t,j^{'}\in m_t}(\dot{\tilde\theta}_{i^{'}}(t)-\dot{\tilde\theta}_{j^{'}}(t)). \]
At each $t$, there are six possible combinations for the  indices ($i^{'}$ and $j^{'}$ ) in the upper Dini derivative of the phase diameter function. We discuss each of them below. 
\begin{lemma}\label{lemderi}Assume $K_1=-K_2<0$.
    Let $t_1,t_2$ be any two real  number such that $0\le t_1<t_2$.  
    \begin{enumerate}
        \item[$(1)$] If $(i^{'},j^{'})\equiv (3,1)$ or $(1,2)$ or $(2,3)$,  for all $t\in (t_1, t_2)$, then
    \[\dot{\tilde\theta}_{i^{'}}(t)-\dot{\tilde\theta}_{j^{'}}(t)=-\frac{2K_2}{3}\cos\left(\frac{\mathcal{D}(\tilde\Theta)}{2}+ \frac{\pi}{6}\right)\left[ 2\sin\left(\frac{\mathcal{D}(\tilde\Theta)}{2}+ \frac{\pi}{6}\right)-\cos\left(\tilde{\theta}_k -\frac{\tilde{\theta}_{i^{'}}+\tilde{\theta}_{j^{'}}}{2}\right) \right]; \]
    \item[$(2)$] If $(i^{'},j^{'})\equiv (2,1)$ or $(3,2)$ or $(1,3)$,  for all $t\in (t_1, t_2)$, then 
    \[\dot{\tilde\theta}_{i^{'}}(t)-\dot{\tilde\theta}_{j^{'}}(t)=-\frac{2K_2}{3}\left[ \sin\left(\mathcal{D}(\tilde\Theta)-\frac{\pi}{3}\right)+\cos\left(\tilde{\theta}_k -\frac{\tilde{\theta}_{i^{'}}+\tilde{\theta}_{j^{'}}}{2}\right)\sin\left(\frac{\mathcal{D}(\tilde\Theta)}{2}+\frac{\pi}{3}\right)\right], \]
    \end{enumerate}
    where $k\neq i^{'},j^{'}.$
\end{lemma}
\begin{proof}
We only need to prove the result for $(i^{'},j^{'})=(3,1)$, and the proofs for the rest are similar. By the definition of $i'$ and $j'$, $\mathcal{D}(\tilde{\Theta}) = \tilde{\theta}_3-\tilde{\theta}_1$. Then
    \begin{align*}
       &\quad \dot{\tilde{\theta}}_3-\dot{\tilde{\theta}}_1\\
       &=\frac{K_2}{3}\sin(\tilde{\theta}_1-\tilde{\theta}_3-\frac{\pi}{3})+\frac{K_2}{3}\sin(\tilde{\theta}_2-\tilde{\theta}_3+\frac{\pi}{3})-\frac{K_2}{3}\sin(\tilde{\theta}_3-\tilde{\theta}_1+\frac{\pi}{3})-\frac{K_2}{3}\sin(\tilde{\theta}_2-\tilde{\theta}_1-\frac{\pi}{3})\\
       &=-\frac{2K_2}{3}\sin\left(\mathcal{D}(\tilde\Theta)+\frac{\pi}{3} \right)+\frac{2K_2}{3}\cos\left(\frac{\tilde{\theta}_2-\tilde{\theta}_3}{2}+\frac{\tilde{\theta}_2-\tilde{\theta}_1}{2}  \right)\sin\left(-\frac{\mathcal{D}(\tilde\Theta)}{2}+\frac{\pi}{3}\right).
    \end{align*}
    Using 
    $\sin\left(-\frac{\mathcal{D}(\tilde\Theta)}{2}+\frac{\pi}{3}\right)=\cos\left(\frac{\mathcal{D}(\tilde\Theta)}{2}+\frac{\pi}{6}\right) $
    and $\sin 2x=2\sin x\cos x$, we obtain that
    \begin{align*}
        \dot{\tilde{\theta}}_3-\dot{\tilde{\theta}}_1&=-\frac{2K_2}{3}\sin\left(\mathcal{D}(\tilde\Theta)+\frac{\pi}{3} \right)+\frac{2K_2}{3}\cos\left(\frac{\tilde{\theta}_2-\tilde{\theta}_3}{2}+\frac{\tilde{\theta}_2-\tilde{\theta}_1}{2}  \right)\cos\left(\frac{\mathcal{D}(\tilde\Theta)}{2}+\frac{\pi}{6}\right)\\
        &=-\frac{2K_2}{3}\cos\left(\frac{\mathcal{D}(\tilde\Theta)}{2}+ \frac{\pi}{6}\right)\left[ 2\sin\left(\frac{\mathcal{D}(\tilde\Theta)}{2}+ \frac{\pi}{6}\right)-\cos\left(\frac{\tilde{\theta}_2-\tilde{\theta}_3}{2}+\frac{\tilde{\theta}_2-\tilde{\theta}_1}{2}\right) \right].
    \end{align*}
\end{proof}}

The following lemma states that the boundedness of the phase diameter implies the exponential decay for the proper upper bound $\frac{2\pi}{3}$.
\begin{lemma}\label{prelem} 
Assume $K_1=-K_2<0$.
For any $T\in (0,\infty]$, if
    $\sup_{t\in[0,T)}\mathcal{D}(\tilde{\Theta}(t))<\frac{2\pi}{3}$ in system \eqref{equtil},
    then there exists $\lambda>0$ such that 
    $ \mathcal{D}(\tilde{\Theta}(t))\le \mathcal{D}(\tilde{\Theta}(0))e^{-\lambda t},$ for any $ t\in[0,T). $
\end{lemma}
\begin{proof}
    Let  $\delta$ be an arbitrary positive number such that $\delta\le \frac{2\pi}{3}-\sup_{t\in[0,T)}\mathcal{D}(\tilde{\Theta}(t))$, then  we shall give 
    the estimation for $D^+\mathcal{D}(\tilde{\Theta}(t))$ under this condition.

    For any $t\in [0,T)$, if $D^+\mathcal{D}(\tilde{\Theta}(t))$ corresponds to Case (1)  in   Lemma \ref{lemderi}, then by equations
    $\cos\left(\frac{\mathcal{D}(\tilde\Theta)}{2}+ \frac{\pi}{6}\right)\ge \sin\frac{\delta}{2}$, $ 2\sin(\frac{x}{2}+\frac{\pi}{6})\ge \frac{x}{5}+1, \forall x\in  [0,\pi], $
    and
    {$\cos\left( \tilde{\theta}_k-\frac{\tilde{\theta}_{i^{'}}+\tilde{\theta}_{j^{'}}}{2}\right)\le 1, k\neq i^{'},j^{'}$},
      we obtain that at time $t$,
    \[D^+\mathcal{D}(\tilde{\Theta}(t))\le -\frac{2K_2}{15}\left(\sin\frac{\delta}{2}\right)\mathcal{D}(\tilde{\Theta}(t)).  \]
     If $D^+\mathcal{D}(\tilde{\Theta}(t))$ corresponds to Case (2) in Lemma \ref{lemderi}, then by {$ \cos\left( \tilde{\theta}_k-\frac{\tilde{\theta}_{i^{'}}+\tilde{\theta}_{j^{'}}}{2}\right)\ge \cos\frac{\mathcal{D}(\tilde{\Theta})}{2}, k\neq i^{'},j^{'}$}, and
    $ \sin(x-\frac{\pi}{3})+\cos\frac{x}{2}\sin(\frac{x}{2}+\frac{\pi}{3})\ge \frac{x}{5},\; \forall x\in  [0,\pi], $
      we obtain that at time $t$,
    \[D^+\mathcal{D}(\tilde{\Theta}(t))\le -\frac{2K_2}{15}\mathcal{D}(\tilde{\Theta}(t)).  \]
    Hence, for any $t\in [0,T)$, $D^+\mathcal{D}(\tilde{\Theta}(t))\le -\frac{2K_2}{15}\left(\sin\frac{\delta}{2}\right)\mathcal{D}(\tilde{\Theta}(t))$ always holds.
    The desired result can be obtained by choosing $\lambda=\frac{2K_2}{15}\sin\frac{\delta}{2}$ and applying the Gr\"{o}nwall's inequality.
\end{proof}
\begin{remark}
To apply Lemma \ref{lemderi}, we need the max-min indices
$i',j'$ to remain constant 
over an interval $(t_1,t_2)$.  For strict rigor in the above proof, 
we can partition $[0,T)$ into 
a countable of sub-intervals to satisfy such conditions by the discontinuous location of the index.
\end{remark}

\begin{proposition}\label{ThmStar5}
Assume $K_1=-K_2<0$.
If the initial configurations of system \eqref{equtil} satisfy
     $\mathcal{D}(\tilde\Theta(0))<\frac{2\pi}{3},$ then there exists $\hat{\lambda}>0$ such that
    $\mathcal{D}(\tilde{\Theta}(t))\le \mathcal{D}(\tilde{\Theta}(0))e^{-\hat{\lambda}t},$ for any $ t\ge 0.  $
\end{proposition}
\begin{proof}
    Select an arbitrary positive number $\hat{\delta}$  such that $\hat{\delta}<\frac{2\pi}{3}-\mathcal{D}(\tilde\Theta(0))$.
    Define the set
        $\mathcal{T}:=\left\{T>0\mid\mathcal{D}(\tilde\Theta(t))<\mathcal{D}(\tilde\Theta(0))+\hat{\delta} ,\; \forall t\in[0,T) \right\}.$
        Obviously, $\mathcal{T}$ is not a empty set and ${T}^*:=\sup\mathcal{T}$ is well defined. We claim that ${T}^*=\infty$. If this is not true, i.e., ${T}^*<\infty$, then
        \begin{equation}\label{contr2}
             \mathcal{D}(\tilde\Theta(t))<\mathcal{D}(\tilde\Theta(0))+\hat{\delta} ,\; \forall t\in[0,{T}^*)\quad \text{and}\quad \mathcal{D}(\tilde\Theta({T}^*))=\mathcal{D}(\tilde\Theta(0))+\hat{\delta}.
        \end{equation}
        By the first assertion of \eqref{contr2} and  Lemma \ref{prelem}, we have that
        \[ \sup_{t\in[0,{T}^*)} \mathcal{D}(\tilde\Theta(t))\le \mathcal{D}(\tilde\Theta(0))+\hat{\delta}<\frac{2\pi}{3}, \] and there exists a constant $\hat{\lambda}>0$ such that
        $\mathcal{D}(\tilde{\Theta}(t))\le \mathcal{D}(\tilde{\Theta}(0))e^{-\hat{\lambda}t},$ $ t\in[0,{T}^*). $
        We let $t\to {T}^{*-}$, and then 
        $\mathcal{D}(\tilde{\Theta}({T}^*)))\le \mathcal{D}(\tilde{\Theta}(0))e^{-\hat{\lambda}{T}^*} \le \mathcal{D}(\tilde{\Theta}(0)), $
        which contradicts the second assertion of  \eqref{contr2}. So, we conduce that ${T}^*=\infty$. Hence, 
         $ \sup_{t\ge 0} \mathcal{D}(\tilde\Theta(t))\le \mathcal{D}(\tilde\Theta(0))+\hat{\delta}<\frac{2\pi}{3}. $
         Applying Lemma \ref{prelem} again, the desired result is obtained.
\end{proof}

{We now provide the main result of this paper.
\begin{theorem}\label{Thm5Star}
    Assume $K_1=-K_2<0$.
    \begin{enumerate}
        \item[$(1)$] If the initial configurations of system \eqref{mod1} satisfy
    \begin{equation}\label{initial5}
        -\pi<\theta_1(0)-\theta_3(0)<\frac{\pi}{3},\; -\frac{\pi}{3}<\theta_2(0)-\theta_3(0)<\pi,\; -\frac{4\pi}{3}<\theta_1(0)-\theta_2(0)<0,
    \end{equation}
    then the solution $\Theta(t)$ converges exponentially fast to the synchronization mode $\Theta^*_5$.
    \item[$(2)$] If the initial configurations of system \eqref{mod1} satisfy
    \begin{equation}\label{initial6}
        -\frac{7\pi}{3}<\theta_1(0)-\theta_3(0)<-\pi,\; -\pi<\theta_2(0)-\theta_3(0)<\frac{\pi}{3},\; -2\pi<\theta_1(0)-\theta_2(0)<-\frac{2\pi}{3},
    \end{equation}
    then the solution $\Theta(t)$ converges exponentially fast to the synchronization mode $\Theta^*_6$.
    \end{enumerate}
\end{theorem}}
\begin{proof}
    Part $(1)$: Based on Proposition \ref{ThmStar5} and equations $\theta_1=\tilde{\theta}_1,{\theta}_2=\tilde\theta_2+\frac{2\pi}{3},{\theta}_3=\tilde\theta_3+\frac{\pi}{3},$ we obtain the first assertion.
    The proof of Part $(2)$ is similar to $(1)$. 
\end{proof}

{\section{Conclusion}\label{sec4}}
We   discussed theoretical findings on the stability and the attraction regions of critical points in Kuramoto oscillators, which relate to the nonlinear dynamics of three flames interacting in an isosceles configuration. The parameter of coupling strength can take negative values. The co-existence of stable oscillators was analyzed with the adequate conditions for their attraction basins.  These findings enhance our qualitative understanding of the complex behaviors in Kuramoto oscillators, even in a minimal system comprising only three nodes.



 {\section*{Acknowledgements}}
 This work is partially supported by Hong Kong RGC GRF grants  11308121, 11318522 and 11308323.
 Zhao acknowledges the  supported of the Hong Kong Scholars Scheme (Grant No. XJ2023001), the Natural Science Foundation of China (Grant No. 12201156), the China Postdoctoral Science Foundation (Grant No. 2021M701013). We thank Peng Zhang at City University of Hong Kong for introducing us the background of flame oscillation.

\bibliographystyle{plain}

\bibliography{ref}

\begin{thebibliography}{10}

\bibitem{antonsen2008external}
T.-M. Antonsen and R.-T. Faghih, et~al.
\newblock External periodic driving of large systems of globally coupled phase
  oscillators.
\newblock {\em Chaos}, 18(3):037112, 2008.

\bibitem{chi2024synchronization}
Y.~Chi, Z.~Hu, T.~Yang, and P.~Zhang.
\newblock Synchronization modes of triple flickering buoyant diffusion flames:
  Experimental identification and model interpretation.
\newblock {\em Phys. Rev. E}, 109(2):024211, 2024.

\bibitem{gergely2020flickering}
A.~Gergely and B.~S{\'a}ndor, et~al.
\newblock Flickering candle flames and their collective behavior.
\newblock {\em Sci. Rep.}, 10(1):21305, 2020.

\bibitem{ha2012basin}
S.-Y. Ha and M.-J. Kang.
\newblock On the basin of attractors for the unidirectionally coupled
  {Kuramoto} model in a ring.
\newblock {\em SIAM J. Appl. Math.}, 72(5):1549--1574, 2012.

\bibitem{ha2014asymptotic}
S.-Y. Ha, Y.~Kim, and Z.~Li.
\newblock Asymptotic synchronous behavior of {Kuramoto} type models with
  frustrations.
\newblock {\em Netw. Heterog. Media}, 9(1):33--64, 2014.

\bibitem{ha2014large}
S.-Y. Ha, Y.~Kim, and Z.~Li.
\newblock Large-time dynamics of {Kuramoto} oscillators under the effects of
  inertia and frustration.
\newblock {\em SIAM J. Appl. Dyn. Syst.}, 13(1):466--492, 2014.

\bibitem{hsia2020synchronization}
C.-H. Hsia, C.-Y. Jung, B.~Kwon, and Y.~Ueda.
\newblock Synchronization of {Kuramoto} oscillators with time-delayed
  interactions and phase lag effect.
\newblock {\em J. Differ. Equ.}, 268(12):7897--7939, 2020.

\bibitem{kitahata2009oscillation}
H.~Kitahata and J.~Taguchi, et~al.
\newblock Oscillation and synchronization in the combustion of candles.
\newblock {\em J. Phys. Chem. A}, 113(29):8164--8168, 2009.

\bibitem{kuramoto1984chemical}
Y.~Kuramoto.
\newblock {\em Chemical {Turbulence}}.
\newblock Springer, 1984.

\bibitem{li2019convergence}
Z.~Li and X.~Xue.
\newblock Convergence of analytic gradient-type systems with periodicity and
  its applications in {Kuramoto} models.
\newblock {\em Appl. Math. Lett.}, 90:194--201, 2019.

\bibitem{li2020synchronization}
Z.~Li and X.~Zhao.
\newblock Synchronization in adaptive {Kuramoto} oscillators for power grids
  with dynamic voltages.
\newblock {\em Nonlinearity}, 33(12):6624, 2020.

\bibitem{lin2007state}
Z.~Lin, B.~Francis, and M.~Maggiore.
\newblock State agreement for continuous-time coupled nonlinear systems.
\newblock {\em SIAM J. Control Optim.}, 46(1):288--307, 2007.

\bibitem{rogge2004stability}
J.-A. Rogge and D.~Aeyels.
\newblock Stability of phase locking in a ring of unidirectionally coupled
  oscillators.
\newblock {\em J. Phys. A: Math. Gen.}, 37(46):11135, 2004.

\bibitem{zhao2018formation}
X.~Zhao, Z.~Li, and X.~Xue.
\newblock Formation, stability and basin of phase-locking for {Kuramoto}
  oscillators bidirectionally coupled in a ring.
\newblock {\em Netw. Heterog. Media}, 13(2):323--337, 2018.

\end{thebibliography}

\end{document}